\documentclass[11pt]{amsart}

\usepackage{amsmath,amssymb,amsthm}
\usepackage[
  pdfauthor={Hiroshi Arai},
  pdftitle={An Elementary Proof of the Nonexistence of Tarski Monster Groups of Exponent 3},
  pdfsubject={Elementary group theory, exponent 3, infinite groups},
  pdfkeywords={group theory, exponent 3, Burnside groups, Engel conditions, 
               nilpotent groups, Tarski monster groups, elementary proofs, 
               commutator identities, key equation, infinite groups}
]{hyperref}

\makeatletter
\def\@setauthors{%
  \begingroup
  \def\thanks##1{}%
  \trivlist
  \centering\normalfont\@topsep30\p@\relax
  \advance\@topsep by -\baselineskip
  \item[]\large\authors
  \endtrivlist
  \endgroup
}
\makeatother

\makeatletter
\def\@evenhead{\thepage\hfil \scriptsize\shorttitle \hfil}
\def\@oddhead{\hfil \scriptsize\shorttitle \hfil\thepage}
\makeatother

\newtheorem{theorem}{Theorem}
\newtheorem{lemma}[theorem]{Lemma}
\newtheorem{proposition}[theorem]{Proposition}
\newtheorem{corollary}[theorem]{Corollary}
\newenvironment{Remark}{\par\noindent\textbf{Remark.}\ }{\par}

\title[No Tarski Monsters (Exponent Three)]%
{An Elementary Proof of the Nonexistence\\
 of Tarski Monster Groups of Exponent~3}
\author{Hiroshi Arai}
\date{}

\begin{document}

\begin{abstract}It is well known that Tarski monster groups of exponent~3 do not exist. 
Traditional proofs rely on deep structural results, such as the restricted Burnside problem, 
properties of the free Burnside group, or Engel-type identities, 
and involve substantial technical computations. 
In this note we give a fully elementary proof: 
the argument reduces to a single simple identity, 
expressed as a key equation, whose verification requires only a brief calculation.
\end{abstract}

\maketitle

\medskip
The proof of our main result reduces to the verification of the following simple identity,
which we call the key equation.

\section{Key Equation}

\begin{lemma}[Key Equation]
In any group of exponent~$3$, one has
\[
aa^g = a^g a \qquad \text{for all } a,g.
\]
\end{lemma}

\begin{proof}
Since $g^{-1}=g^2$, we compute
\[
aa^g = a g^{-1} a g = (g a^2)^{-1} a g = (g a^2)^2 a g.
\]
Expanding the product yields
\[
g a^2 g a^2 a g = g a^2 g^2 = g(ga)^{-1} = g (ga)^2 = gga g a = g^2 a g a = a^g a,
\]
as claimed.
\end{proof}

\begin{Remark}
The key equation also plays a role in Hall’s classic discussion of the free Burnside group of exponent~$3$ 
(denoted $B(3,r)$ in his notation, as opposed to the now-standard $B(r,3)$). 
There the starting point is the relation $(xy)^3=1$, which can be expanded to yield identities such as
\[
yxy = x^{-1} y^{-1} x^{-1}.
\]
By repeated use of these identities one eventually arrives at
\[
x y^{-1} x y = y^{-1} x y x,
\]
which is an instance of the key equation. 
Thus, while the identity does arise in Hall’s treatment, it emerges only after a sequence of substitutions 
within his calculation of the order of $B(r,3)$.
\end{Remark}

\section{Existence of an Infinite Abelian Subgroup}
\begin{proposition}
Every infinite group of exponent~$3$ has an infinite abelian subgroup.
\end{proposition}

\begin{proof}
By the Key Equation, every element commutes with all of its conjugates.

\medskip
\noindent\textit{Case~1.} There exists $a\in G$ whose conjugacy class $a^G$ is infinite.
Then $\langle a^G\rangle$ is abelian; since it is generated by an infinite set, it is infinite.

\medskip
\noindent\textit{Case~2.} All conjugacy classes in $G$ are finite.
Then $|G:C_G(x)|<\infty$ for each $x\in G$, and hence the intersection of finitely many
such centralizers has finite index.

Start with any $x_1\in G$, and set $H_1=C_G(x_1)$.
Inductively, having chosen commuting elements $x_1,\dots,x_{n-1}$, put
\[
H_n=\bigcap_{i=1}^{n-1} C_G(x_i).
\]
Then $H_n$ has finite index in $G$, and every element of $H_n$ commutes with each $x_i$.

If $H_n=\langle x_1,\dots,x_{n-1}\rangle$, then this abelian subgroup has finite index in $G$,
hence necessarily infinite, and we are done. Otherwise, choose
\[
x_n\in H_n\setminus \langle x_1,\dots,x_{n-1}\rangle.
\]
By construction, the $x_i$ are pairwise commuting and distinct. Continuing indefinitely
produces an infinite set of commuting elements, which generates an infinite abelian subgroup.
\end{proof}

\begin{Remark}
Note that the inductive choice in Case~2 is also immediate 
if one uses the standard fact that a finitely generated abelian group of exponent~$3$ is finite. 
More generally, the proposition can also be deduced from the restricted Burnside problem
or from the classification of $2$-Engel groups of exponent~$3$. 
We deliberately avoid all of these inputs here in order to keep the argument completely elementary.
\end{Remark}

\section{Nonexistence of Tarski Monsters for \texorpdfstring{$p=3$}{p=3}}

A \emph{Tarski monster group} for a prime~$p$ is an infinite group in which every proper, nontrivial subgroup has order~$p$. 
The following corollary is an immediate consequence of the preceding proposition:

\begin{corollary}
There is no Tarski monster group of exponent~$3$.
\end{corollary}

\begin{proof}
If $G$ were such a group, then $G$ would be infinite of exponent~$3$, 
but every proper nontrivial subgroup of $G$ would be cyclic of order~$3$. 
By the proposition, however, $G$ must have an infinite abelian subgroup, which is a contradiction.
\end{proof}

\begin{Remark} 
The nonexistence of Tarski monster groups of exponent~$3$ is, of course, already known. 
Classical arguments reach this conclusion by combining the fact with other deep structural results. 
In such approaches the proof itself may be short, but it ultimately relies on heavy prerequisites. 
By contrast, the present proof is entirely elementary: 
both the key identity (the Key Equation) 
and its application to the argument are established by direct and simple calculations. 
\end{Remark}

\section*{Closing Note}
It is well known that Tarski monster groups of exponent~3 do not exist, 
and this fact can be established in several ways. 
Traditional arguments invoke the restricted Burnside problem, 
properties of the free Burnside group~$B(r,3)$, or Engel-type identities. 
While effective, these methods ultimately depend on sophisticated structural results 
and involve substantial technical computations. 
In this note we present a fully elementary proof: 
the argument reduces to a single simple identity, 
expressed as a key equation, whose verification requires only a brief calculation.
This short note was prepared with educational motivation in mind, 
and earlier drafts were prepared under the name ``Sulaxia,'' 
which the author sometimes uses for expository work. 
It is hoped that this elementary presentation may be of interest both to researchers 
and to students encountering this circle of ideas.



\begin{thebibliography}{99}

\bibitem{Arai19}
H.~Arai,
\emph{An elementary proof of the existence of an infinite abelian subgroup of an infinite group of exponent~3},
Bulletin of Kochi University of Technology, Series on Mathematics
\textbf{16} (No.~1) (2019), 209--214 (in Japanese).
Available at \url{https://kutarr.kochi-tech.ac.jp/records/1846}.

\bibitem{Burnside05}
W.~Burnside,
\emph{On groups of exponent three},
Transactions of the Cambridge Philosophical Society
\textbf{20} (1905), 251--282.

\bibitem{Hall}
M.~Hall,
\emph{The Theory of Groups},
Macmillan, New York, 1959.

\end{thebibliography}
\end{document}